\documentclass[english]{amsart}

\usepackage{adjustbox}
\usepackage[latin9]{inputenc}
\usepackage{hyperref}
\usepackage{lscape}
\usepackage{slashbox}
\usepackage{mathrsfs}  
\usepackage{enumerate}  

\makeatletter
 \theoremstyle{plain}
\newtheorem{thm}{Theorem}[section]
\theoremstyle{plain}
  \newtheorem{prop}[thm]{Proposition}
\theoremstyle{plain}
 \newtheorem{lemma}[thm]{Lemma}
\theoremstyle{plain}

\theoremstyle{plain}

\theoremstyle{definition}
  
\theoremstyle{definition}
  
 \theoremstyle{definition}

\theoremstyle{remark}
\newtheorem{rmk}[thm]{Remark}
\numberwithin{equation}{section}
 
\usepackage{amsmath}
\usepackage{amssymb}
\usepackage{amscd}
\usepackage{amsthm}
\usepackage{array,multirow}
\usepackage[arrow,matrix,tips,all,cmtip,poly]{xy}

\usepackage{stmaryrd}
\usepackage{url}
\usepackage{color}
\usepackage{caption}
\usepackage{float}

\newcommand{\Z}{\mathbb{Z}}

\newcommand{\Q}{\mathbb{Q}}

\newcommand{\F}{\mathbb{F}}

\newcommand{\Hom}{\mathrm{Hom}}

\newcommand{\End}{\mathrm{End}}

\newcommand{\GL}{\mathrm{GL}}

\newcommand{\id}{\mathrm{id}}

\newcommand{\tld}[1]{\widetilde{#1}}

\newcommand{\soc}{\mathrm{soc}}

\newcommand{\ovl}[1]{\overline{#1}}

\newif\iffinalrun
\iffinalrun
  \newcommand{\mar}[1]{}
\else
  \newcommand{\mar}[1]{\marginpar{\raggedright\tiny #1}}

\newcommand{\ra}{\rightarrow}

\newcommand{\into}{\hookrightarrow}

\newcommand{\risom}{\buildrel\sim\over\rightarrow}

\title{On some nonadmissible smooth irreducible representations for $\GL_2$}
\author{Daniel Le}
\begin{document}

\maketitle

\begin{abstract}
Let $p>2$ be a prime.
We give examples of smooth absolutely irreducible representations of $\GL_2(\Q_{p^3})$ over ${\F}_{p^3}$ which are not admissible.
\end{abstract}

\section{Introduction}

Smooth representations of $p$-adic reductive groups arise naturally in the theory of automorphic forms.
Smooth here means that every vector is invariant under an open subgroup.
Classical finite-dimensionality results for automorphic forms imply \emph{admissibility}: the invariants of the representation under any open subgroup is finite-dimensional.
Both of these notions make sense for a base field of any characteristic.
Representation theory over base fields of positive characteristic has attracted considerable attention in recent years because of its connection to congruences of automorphic forms and the modularity of Galois representations.

In the recent groundbreaking work \cite{AHHV}, smooth, irreducible, admissible mod $p$ representations of connected reductive $p$-adic groups are classified in terms of supercuspidal representations, closely mirroring the earlier theory in characteristic not equal to $p$.
For a base field of characteristic different from $p$, it is known from \cite[II.2.8]{V} moreover that every smooth irreducible representation of a connected reductive $p$-adic group is admissible.
\cite[Question 1]{AHHV2} asks whether a similar statement is true for mod $p$ representations.
As mentioned in \emph{loc.~cit.}, this question has an affirmative answer in some simple cases and when the group is $\GL_2(\Q_p)$ combining results of \cite{BL,Br03,Berger}.
We provide a negative answer, at least when $p>2$, already for $\GL_2$ but over a larger field.

\begin{thm}\label{thm:main}
Let $p>2$.
There exists a smooth absolutely irreducible $\GL_2(\Q_{p^3})$-representation over $\F_{p^3}$ which is not admissible.
\end{thm}

\noindent 
It will be clear from the construction that there are infinitely many such representations.
Moreover, similar constructions exist for unramified extensions of larger degree (see Remark \ref{rmk:f=3}), but we content ourselves with describing the simplest example.
The above result is yet another example of a distinguishing feature of the mod $p$ theory, namely that the theory is very sensitive to the field of definition of the group.

Admissibility is a desirable property, in part because it implies that the irreducible representation has a central character, admits Hecke eigenvalues for weights, and has an endomorphism ring of finite dimension over the base field.
\cite[Question 2, Question 8]{AHHV2} ask whether irreducible mod $p$ representations must have central characters and Hecke eigenvalues.
The representations that we construct have central characters and Hecke eigenvalues (matching certain supersingular representations), and so we do not answer these questions.
However, by restricting scalars for a representation we construct, we also prove the following.

\begin{thm}\label{thm:schur}
There exists a smooth irreducible $\GL_2(\Q_{p^3})$-representation over $\F_{p^3}$ whose endomorphisms contain $\ovl{\F}_p$.
\end{thm}

\noindent Of course, such a representation cannot be absolutely irreducible as the endomorphism ring over $\ovl{\F}_p$ would contain $\ovl{\F}_p\otimes_{\F_{p^3}} \ovl{\F}_p$.

We now make brief remarks on the construction.
Irreducible mod $p$ representations are typically rather difficult to construct, much less nonadmissible ones.
Global constructions coming from the theory of automorphic forms always give admissible representations and parabolic induction preserves admissibility.
However, the Bruhat--Tits tree and the diagrams of \cite{Paskunas} give a powerful method of constructing mod $p$ representations of $p$-adic $\GL_2$ with fixed $K$-socle where $K$ is the maximal compact subgroup.
\cite{BP} uses this close control of the $K$-socle to prove both irreducibility and admissibility for many representations that they construct.
The main idea of this paper is that the control of the $K$-socle can also be used to prove irreducibility and \emph{nonadmissibility}. 
We construct an infinite-dimensional diagram that gives rise to a nonadmissible $\GL_2(\Q_{p^3})$-representation, and prove irreducibility using the methods of \emph{ibid.}

\subsection{Acknowledgments}

We thank Marie-France Vigneras for encouraging the author to work on this problem.
The intellectual debt owed to the work of Christophe Breuil and Vytautas Pa{\v{s}}k{\=u}nas will be clear to the reader.
We thank Stefano Morra for producing some diagrams in an earlier version.
We thank Benjamin Schraen for suggesting that we could prove Theorem \ref{thm:schur}.
We thank Florian Herzig, Yongquan Hu, Karol Koziol, and an anonymous referee for many helpful comments on an earlier version.
The author was supported by the Simons Foundation under an AMS-Simons travel grant and by the National Science Foundation under the Mathematical Sciences Postdoctoral Research Fellowship No.~ 1703182.
We thank the University College London, ENS de Lyon, and L'Institut Montpelli\'erain Alexander Grothendieck for providing hospitality and excellent working conditions while part of this work was carried out.

\subsection{Notation}\label{subsec:not}
Let $p>2$ and let $q$ be $p^f$ for a positive integer $f$.
Fix an algebraic closure $\ovl{\F}_p$ of $\F_q$.
If $V$ is an $\F_q$-vector space, let $V_{\ovl{\F}_p}$ denote $V\otimes_{\F_q}\ovl{\F}_p$

Let $G$ be $\GL_2(\Q_q)$, $Z$ the center of $G$, $K$ be $\GL_2(\Z_q)$, and $I$ (resp.~ $I_1$) the preimage in $K$ of the upper triangular matrices (resp.~unipotent upper triangular matrices) in $\GL_2(\F_q)$ under the natural reduction map.
Let $\Pi\in G$ be the matrix $\Bigl(\begin{smallmatrix} 0 & 1 \\ p &0 \end{smallmatrix} \Bigr)$.
Then $\Pi$ normalizes $I$ and the normalizer $N(I)$ of $I$ is $IZ \sqcup IZ \Pi$.
Moreover, we have an isomorphism
\begin{align}\label{eqn:norm}
N(I)\Big/\Big\langle \Bigl(\begin{smallmatrix} p & 0 \\ 0 & p \end{smallmatrix} \Bigr)\Big\rangle &\risom I \rtimes \Z/2 \\
\Pi &\mapsto (\id,1).
\end{align}
For a character $\chi$ of $IZ$, let $\chi^s$ be the character of $IZ$ given by precomposing $\chi$ by $\Pi$-conjugation.
If $V$ is an $IZ$-representation and $\chi$ a character of $IZ$, we let $V^\chi$ be the $\chi$-isotypic part of $V$.

\section{Diagrams} \label{sec:diag}

\subsection{Diamond diagrams} \label{subsec:dd}

A \emph{diagram} is a triple $(D_0,D_1,r)$ where $D_0$ is a smooth $KZ$-representation, $D_1$ is a smooth $N(I)$-representation, and $r$ is an $IZ$-equivariant map $D_1 \ra D_0$.
A diagram is a \emph{basic $0$-diagram} if $r$ induces an isomorphism $D_1 \risom D_0^{I_1}$.

Let $\rho:G_{\Q_q} \ra \GL_2(\ovl{\F}_p)$ be a generic continuous irreducible representation in the sense of \cite[Definition 11.7]{BP} (such representations exist with the assumption that $p>2$).
Let $\mathcal{D}(\rho)$ be the set of Serre weights defined in \cite[\S 11]{BP}.
To $\rho$, \cite[Theorem 13.8]{BP} attaches a family of basic $0$-diagrams.
We fix for the rest of the paper a basic $0$-diagram $(D_0(\rho),D_1(\rho),r)$ in this family which is defined over $\F_q$.
That is $D_0(\rho)$ and $D_1(\rho)$ are finite dimensional $KZ$ and $N(I)$-representations over $\F_q$, respectively, and $(D_0(\rho)_{\ovl{\F}_p},D_1(\rho)_{\ovl{\F}_p},r)$ is a member of the family constructed in \emph{loc.~cit.}
Then $r$ identifies $D_1(\rho)$ with $D_0(\rho)^{I_1}$ as $IZ$-representations, which we will identify implicitly.

In fact, the isomorphism classes of $D_0(\rho)$ and $D_1(\rho)$ do not depend on the above choice (though $r$ does). 
The $K$-representation $D_0(\rho)|_K$ satisfies the following properties:
\begin{itemize}
\item the $K$-action on $D_0(\rho)$ factors through $\GL_2(\F_q)$;
\item there is a direct sum decomposition \[D_0(\rho) = \oplus_{\sigma \in \mathcal{D}(\rho)} D_{0,\sigma}(\rho)\]
where the $\GL_2(\F_q)$-socle of $D_{0,\sigma}(\rho)$ is $\sigma$ for all $\sigma \in \mathcal{D}(\rho)$;
\item the Jordan--H\"older factors of $D_0(\rho)$ are multiplicity free (\cite[Theorem 13.8]{BP}), and $D_1(\rho)$ is a multiplicity free semisimple $IZ$-representation (\cite[Lemma 14.1]{BP}).
\end{itemize}

Recall from \cite[Lemma 11.4]{BP} and the paragraph thereafter that there is a bijection 
\begin{align*}
2^{\Z/f} &\ra \mathcal{D}(\rho) \\
J &\mapsto \sigma_J
\end{align*}
Define an automorphism $\delta: 2^{\Z/f} \ra 2^{\Z/f}$ by $j\in \delta(J)$ if and only if $j+1 \in J$ (resp.~$j+1\notin J$) for $j \neq 0$ (resp.~for $j=0$).
This ``shift then flip at $j=0$" is denoted $\delta_i$ in \cite[\S 15]{BP}.

We introduce one final piece of notation.
For $0 \leq s \leq q-1$, let 
\[
S_s := \sum_{\lambda \in \F_q} \lambda^s \Bigl(\begin{smallmatrix} [\lambda] & 1 \\ 1 &0 \end{smallmatrix} \Bigr) \in \F_q[K].
\]

\begin{prop}\label{prop:Sop}
Let $v$ be a nonzero element in $D_1(\rho)^\chi$.
Then there is a unique $0\leq s(\chi)\leq q-1$ such that $S_{s(\chi)}(v)$ is a nonzero element of $(\soc_K D_0(\rho))^{I_1}$.
\end{prop}
\begin{proof}
Since $D_1(\rho)^\chi$ is one-dimensional, the $K$-representation generated by $v$ has irreducible socle using the last two bulleted points above (cf.~the proof of \cite[Proposition 5.1(i)]{Breuil}).
The result now follows from \cite[Lemma 2.7]{BP}.
\end{proof}

Define a linear map 
\[
S: D_1(\rho) \ra (\soc_K D_0(\rho))^{I_1}\\
\]
which maps a nonzero $(IZ,\chi)$-eigenvector $v$ to $S_{s(\chi)}v$.

We recall the following result.

\begin{prop}\label{prop:cycle}
Let $\chi_J$ be the $I$-character of $\sigma_J^{I_1}$.
Then $S\circ \Pi$ gives an isomorphism $D_1(\rho)^{\chi_J}$ to $D_1(\rho)^{\chi_{\delta(J)}}$ for all $J \in 2^{\Z/f}$.
\end{prop}
\begin{proof}
This follows from \cite[Lemma 15.2]{BP} (see also the proof of \cite[Proposition 5.1]{Breuil}).
\end{proof}

\subsection{An infinite diagram}

In this section, we let $f$ be $3$.
\begin{rmk}\label{rmk:f=3}
When $f=2$, $2^{\Z/f}$ is a single $\delta$-orbit.
When $f=3$, $2^{\Z/f}$ consists of two $\delta$-orbits, namely
\[
\emptyset \mapsto \{0\} \mapsto \{0,2\} \mapsto \{0,1,2\} \mapsto \{1,2\} \mapsto \{1\} \mapsto \emptyset 
\]
\[
\{2\} \mapsto \{0,1\} \mapsto \{2\}.
\]
For $f>3$, $2^{\Z/f}$ always contains more than one $\delta$-orbit since the size of an orbit must divide the order of the automorphism $\delta$, which is $2f$.
It is the existence of more than one $\delta$-orbit which allows us to make the construction in this section.
\end{rmk}
Let $D_0$ be the $KZ$-representation $\oplus_{i\in \Z} D_{0,i}$ where there is a fixed isomorphism $D_{0,i}\cong D_0(\rho)_{\ovl{\F}_p}$.
Let $\iota_i$ be the inclusion $D_0(\rho) \subset D_{0,i} \subset D_0$.
For $v\in D_0(\rho)$, we denote $\iota_i(v)$ by $v_i$.

Let $D_1$ be $D_0^{I_1}$.
Let $\lambda = (\lambda_i)_{i\in \Z}$ be in $\prod_{i\in \Z}\ovl{\F}_p^\times$.
For such a $\lambda$, we now define an action of $N(I)$ on $D_1$ such that $\Pi^2$ acts trivially.
By (\ref{eqn:norm}), it suffices to define an involution on $D_1$ taking $D_1^\chi$ to $D_1^{\chi^s}$ for every character $\chi$ of $IZ$.
We will denote this involution by $\tld{\Pi}$.

Let $\chi_+$ be the $IZ$-character of the space $\sigma_{\{1\} }^{I_1}$ and $\chi_-$ be the $IZ$-character of the space $\sigma_{\{0,1\} }^{I_1}$ (as usual $\Pi^2$ acts trivially).

\begin{prop}\label{prop:chi}
There is an $IZ$-character $\chi_1$ (resp.~$\chi_2$) such that both of the spaces $D_{0,\sigma_{\{2\} }}(\rho)^{\chi_1}$ and $D_{0,\sigma_\emptyset}(\rho)^{\chi_1^s}$ \emph{(}resp.~$D_{0,\sigma_{\{0,1\} }}(\rho)^{\chi_2}$ and $D_{0,\sigma_{\{0\} }}(\rho)^{\chi_2^s}$\emph{)} are nonzero.
\end{prop}
\begin{proof}
This follows from an explicit check using \cite[Corollary 14.10 and Lemma 15.2]{BP}.
In the notation of \cite[\S 11]{BP}, we have that $\sigma_{\{2\} }$ corresponds to 
\[
(\lambda_0(r_0),\lambda_1(r_1),\lambda_2(r_2)) = (r_0, p-2-r_1,r_2+1)
\]
and $\sigma_{\{0,1\} }$ corresponds to
\[
(\lambda_0(r_0),\lambda_1(r_1),\lambda_2(r_2)) = (p-1-r_0,r_1+1,p-2-r_2).
\]
Then $\chi_1$ corresponds to 
\[
(\mu_0(\lambda_0(r_0)),\mu_1(\lambda_1(r_1)),\mu_2(\lambda_2(r_2))) = (p-2-r_0, p-1-r_1,r_2+1)
\]
and $\chi_2$ corresponds to 
\[
(\mu_0(\lambda_0(r_0)),\mu_1(\lambda_1(r_1)),\mu_2(\lambda_2(r_2))) = (p-r_0, r_1+1,r_2).
\]
\end{proof}

In fact, the characters $\chi_1$ and $\chi_2$ are uniquely described by the properties in Proposition \ref{prop:chi}, but we will not use this.
As we will see, the only property that we will need is that $\chi_1$ (resp.~$\chi_2$) is a character in $D_{0,\sigma_{\{2\} }}(\rho)^{I_1}$ (resp.~$D_{0,\sigma_{\{0,1\} }}(\rho)^{I_2}$), which is not in $(\sigma_{\{2\} })^{I_1}$ (resp.~$(\sigma_{\{0,1\} })^{I_1}$).
The exact choices and formulas of Propposition \ref{prop:chi} will not be important, and we include them only for the sake of concreteness.

If $v\in D_1(\rho)^\chi$ with 
\[\chi\notin \{ \chi_+, \chi_+^s, \chi_-, \chi_-^s, \chi_1, \chi_1^s\},\]
we define
\[\tld{\Pi}(v_i) = (\Pi v)_i.\]
If $v\in D_1(\rho)^{\chi_+}$, then we define
\[\tld{\Pi}(v_i) = (\Pi v)_{i+1}.\]
If $v\in D_1(\rho)^{\chi_-}$, then we define
\[\tld{\Pi}(v_i) = (\Pi v)_{i-1}.\]
If $v\in D_1(\rho)^{\chi_1}$, then we define
\[\tld{\Pi}(v_i) = \lambda_i(\Pi v)_i.\]
This now uniquely defines an $\ovl{\F}_p$-linear involution $\tld{\Pi}$ of $D_1$, and it takes $D_1^\chi$ to $D_1^{\chi^s}$ for every character $\chi$ of $IZ$ as desired.

Let $D(\lambda)$ be the basic $0$-diagram $(D_0, D_1, \mathrm{can})$ with the above actions, where $\mathrm{can}$ denotes the canonical inclusion $D_1\subset D_0$.
We define an $\ovl{\F}_p$-linear map $\tld{S}: D_1 \ra (\soc_K D_0)^{I_1}$ by the formula $\tld{S} \iota_i = \iota_i S$, where $S$ is as defined in \S \ref{subsec:dd}.

\section{The construction}\label{sec:irr}

For the purposes of notation, we review the proof of the following result, which is a special case of \cite[Theorem 9.8]{BP}, although we work over $\F_q$ rather than $\ovl{\F}_p$.

\begin{thm}\label{thm:BPcon}
There exists a smooth $G$-representation $\tau$ over $\F_q$ such that 
\begin{itemize}
\item there is an injection of diagrams $(D_0(\rho),D_1(\rho),r) \subset (\tau|_{KZ},\tau|_{N(I)},\id)$;
\item $\tau$ is generated as a $G$-representation by the image of $D_0(\rho)$; and 
\item the induced injection $\soc_K D_0(\rho) \into \soc_K \tau$ is an isomorphism.
\end{itemize}
\end{thm}
\begin{proof}
Let $\Omega$ be the $K$-injective envelope of $D_0(\rho)|_K$.
We give $\Omega$ a $KZ$-action by demanding that $\Pi^2$ acts trivially.
There is an idempotent $e\in \End_I(\Omega)$ such that $e(\Omega)|_I$ is an $I$-injective envelope of $D_1(\rho)$.
There is a decomposition of $e(\Omega)|_I$ as a direct sum 
\[\oplus_\chi \Omega_\chi,\]
where $\chi$ runs over the $I$-characters in $D_1(\rho)$ and $\Omega_\chi$ is an $I$-injective envelope of the $\chi$-isotypic part of $D_1(\rho)$.
By \cite[Lemma 9.5]{BP}, there is an ${\F}_q$-linear map $e(\Omega) \ra e(\Omega)$ which intertwines the action and $\Pi$-conjugate action of $IZ$, extends the action of $\Pi$ on $D_1(\rho)$, and whose restriction to $\Omega_\chi$ for each $\chi$ above gives a map
\[\Omega_\chi \ra \Omega_{\chi^s}.\]
This gives an action of $N(I)$ on $e(\Omega)$.
There is also an action of $N(I)$ on $(1-e)(\Omega)$ by \cite[Lemma 9.6]{BP}.
This gives an action of $N(I)$ on $\Omega$ whose restriction to $I$ is compatible with the action coming from $KZ$ on $\Omega$.
By \cite[Corollary 5.18]{Paskunas}, this gives an action of $G$ on $\Omega$.
We then take $\tau$ to be the $G$-representation generated by $D_0(\rho)$.
\end{proof}

\begin{thm}\label{thm:con}
There exists a smooth $G$-representation $\pi$ over $\ovl{\F}_p$ such that 
\begin{itemize}
\item there is an injection of diagrams $D(\lambda) \subset (\pi|_{KZ},\pi|_{N(I)},\id)$;
\item $\pi$ is generated as a $G$-representation by the image of $D_0$; 
\item the induced injection $\soc_K D_0 \into \soc_K \pi$ is an isomorphism; and
\item if $\lambda \in \prod_{i\in \Z} \F_q^\times$, then $\pi$ is defined over $\F_q$.\end{itemize}
\end{thm}
\begin{proof}
Let $\Omega$ be the $K$-injective envelope of $D_0(\rho)|_K$ as in the proof of Theorem \ref{thm:BPcon}.
We give $\Omega$ a $KZ$-action by demanding that $\Pi^2$ acts trivially.
Recall the definitions of $e \in \End_I(\Omega)$ and $\Omega_\chi$ from the proof of Theorem \ref{thm:BPcon}.
Now let $\Omega_\infty$ be the $KZ$-representation $\oplus_{i\in \Z} \Omega_i$ where there is a fixed isomorphism $\Omega_i \cong \Omega_{\ovl{\F}_p}$.
Let $\iota_i$ be the $KZ$-injection $\Omega \subset \Omega_i \subset \Omega_\infty$.
To define an action of $N(I)$ on $\Omega_\infty$, it suffices to define an involution, which we call $\tld{\Pi}$, on $\Omega_\infty$ which intertwines the action and $\Pi$-conjugate action of $IZ$.
For each $i\in \Z$, we define $\tld{\Pi} \circ \iota_i|_{(1-e)(\Omega)}$ to be $\iota_i \circ \Pi|_{(1-e)(\Omega)}$.
For $\chi \notin \{\chi_+,\chi_+^s,\chi_-,\chi_-^s,\chi_1,\chi_1^s\}$, we define $\tld{\Pi} \circ \iota_i|_{\Omega_\chi}$ to be $\iota_i \circ\Pi|_{\Omega_\chi}$.
We define $\tld{\Pi} \circ \iota_i|_{\Omega_{\chi_+}}$ to be $\iota_{i+1} \circ \Pi|_{\Omega_{\chi_+}}$, $\tld{\Pi} \circ \iota_i|_{\Omega_{\chi_-}}$ to be $\iota_{i-1} \circ \Pi|_{\Omega_{\chi_-}}$, and $\tld{\Pi} \circ \iota_i|_{\Omega_{\chi_1}}$ to be $\iota_i \circ \lambda_i\Pi|_{\Omega_{\chi_1}}$.
This completely determines the $\ovl{\F}_p$-linear involution $\tld{\Pi}$.
It is easy to see that the defined action of $N(I)$ on $\Omega_\infty$ extends the action of $N(I)$ on $D_1$.
By \cite[Corollary 5.18]{Paskunas}, this gives an action of $G$ on $\Omega_\infty$.
If $\lambda \in \prod_{i\in \Z} \F_q^\times$, then this action is defined over $\F_q$.
Then if we let $\pi$ be the $G$-subrepresentation of $\Omega_\infty$ generated by $D_0$, $\pi$ satisfies the required hypotheses.
Indeed, we have that $\soc_K\Omega_\infty = \soc_K \pi = \soc_K D_0$, and $\pi$ is defined over $\F_q$ if $\Omega_\infty$ is.
\end{proof}

Let $D_{0,I}(\rho)$ and $D_{0,II}(\rho)$ be $D_{0,\sigma_{\{2\} }}(\rho)\oplus D_{0,\sigma_{\{0,1\} }}(\rho)$ and $\oplus_J D_{0,\sigma_J}(\rho)$, respectively, where the sum is over 
\[J \in \{\emptyset, \{0\},\{0,2\}, \{0,1,2\}, \{1,2\}, \{1\}\}.\]
(This partition $2^{\Z/3} = J \cup J^c$ corresponds to $\delta$-orbits, see Remark \ref{rmk:f=3}.)
We now recall the following special case of \cite[Theorem 19.10(i)]{BP}, since the arguments play a crucial role in the proof of Theorem \ref{thm:irr}.

\begin{thm}\label{thm:BPirr}
Any $G$-representation $\tau$ satisfying the hypotheses in Theorem \ref{thm:BPcon} is absolutely irreducible.
\end{thm}
\begin{proof}
Let $\tau'\subset \tau_{\ovl{\F}_p}$ be a nonzero $G$-subrepresentation.
Since $\soc_K \tau_{\ovl{\F}_p} \cong \soc_K D_0(\rho)_{\ovl{\F}_p}$, there is a $J$ such that $\Hom_K(\sigma_J,\tau')$ is nonzero.
Then by \cite[Lemma 19.7]{BP}, we have the inclusion $D_{0,\sigma_{\delta(J)}}(\rho)_{\ovl{\F}_p} \subset \tau'$.
Repeating this, one obtains an inclusion of one of $D_{0,I}(\rho)_{\ovl{\F}_p}$ and $D_{0,II}(\rho)_{\ovl{\F}_p}$ in $\tau'$.
Then either $(\tau')^{I,\chi_1}$ or $(\tau')^{I,\chi_1^s}$ is nonzero.
Applying $\Pi$, we see that they both must be nonzero so that $D_{0,\sigma_\emptyset}(\rho)_{\ovl{\F}_p}$ and $D_{0,\sigma_{\{2\} }}(\rho)_{\ovl{\F}_p}$ are both in $\tau'$.
Repeating the earlier argument, we have that $D_0(\rho)_{\ovl{\F}_p}\subset \tau'$.
Since $\tau_{\ovl{\F}_p}$ is generated by $D_0(\rho)_{\ovl{\F}_p}$, we have that $\tau' = \tau_{\ovl{\F}_p}$.
\end{proof}

The following is the main result of this section.

\begin{thm}\label{thm:irr}
If $\lambda_0\in \F_q$ and $\lambda_i \neq \lambda_0$ for all $i \neq 0$, then any $G$-representation $\pi$ satisfying the hypotheses in Theorem \ref{thm:con} is irreducible over $\ovl{\F}_p$.
If moreover the $\F_q$-span of $(\lambda_i)_i$ is $\ovl{\F}_p$, then $\pi$ is irreducible as a $G$-representation over $\F_q$.
\end{thm}
\begin{proof}
Let $\pi'$ be a nonzero $G$-subrepresentation of $\pi$ seen as a representation over $\F_q$ by restriction of scalars.
Since $\soc_K\pi' \subset D_0$, there exists $\sigma\in \mathcal{D}(\rho)$ such that $\Hom_K(\sigma,\pi')$ is nonzero.
Then there exists a $(c_i)_i$ in $\oplus_{i\in \Z} \ovl{\F}_p$ such that 
\[
\Big( \sum_i c_i \iota_i\Big) (D_{0,\sigma }(\rho)) \cap \pi' \neq 0.
\]

\begin{lemma} \label{lemma:circle}
Suppose that $\sigma\in \mathcal{D}(\rho)$ and $(d_i)_i\in \oplus_{i\in \Z} \ovl{\F}_p$ are elements such that 
\[
\Big( \sum_i d_i \iota_i\Big) (D_{0,\sigma}(\rho)) \cap \pi' \neq 0.
\]
Then for any $j\in \Z$, 
\[
\Big( \sum_i d_i \iota_{i+j}\Big) (D_0(\rho)) \subset \pi'.
\]
\end{lemma}
\begin{proof}
We assume that $\sigma$ is $\sigma_\emptyset$, as the other cases are similar.
Then as in the proof of Theorem \ref{thm:BPirr}, we see from repeatedly applying $\tld{S}\tld{\Pi}$ that 
\[
\Big( \sum_i d_i \iota_{i+j}\Big) (D_{0,II}(\rho)) \subset \pi'
\]
for $j > 0$.
Since for each $j > 0$, we have that 
\[
\Big( \sum_i d_i \iota_{i+j}\Big) (D_{0,II}(\rho)^{\chi_2^s}) \subset \pi',
\]
we have that 
\[
\Big( \sum_i d_i \iota_{i+j}\Big) (D_{0,I}(\rho)^{\chi_2}) \subset \pi'
\]
for $j>0$.
Again repeatedly applying $\tld{S}\tld{\Pi}$, we see that
\[
\Big( \sum_i d_i \iota_{i+j}\Big) (D_{0,I}(\rho)) \subset \pi'
\]
for all $j\in \Z$.
Then since
\[
\Big( \sum_i d_i \iota_{i+j}\Big) (D_{0,I}(\rho)^{\chi_2}) \subset \pi'
\]
for all $j\in \Z$, we have that 
\[
\Big( \sum_i d_i \iota_{i+j}\Big) (D_{0,II}(\rho)^{\chi_2^s}) \subset \pi'
\]
for all $j\in \Z$.
We conclude that 
\[
\Big( \sum_i d_i \iota_{i+j}\Big) (D_{0,II}(\rho)) \subset \pi'
\]
for all $j\in \Z$ by again repeatedly applying $\tld{S}\tld{\Pi}$.
\end{proof}

In the proof of the next lemma, we will use the following notation.
For $(d_i)_i \in  \oplus_{i\in \Z} \ovl{\F}_p$, let $\#(d_i)_i$ be the cardinality of $\{i\in \Z|d_i\neq 0\}$.

\begin{lemma}\label{lemma:one}
There is a nonzero constant $c \in \ovl{\F}_p$ such that $c\iota_0(D_{0,\sigma_{\{0\} }}(\rho)) \subset \pi'$.
\end{lemma}
\begin{proof}
Fix nonzero elements $v^1 \in D_1(\rho)^{\chi_1}$ and $v^2 \in D_1(\rho)^{\chi_2}$.
One checks that $(S \Pi)^2 v^1$ and $S \Pi v^2$ are nonzero elements in $\sigma_{\{ 0 \} }^{I_1} \subset D_1(\rho)$ using the definition of $\chi_1$ and $\chi_2$ and Proposition \ref{prop:cycle}.
Thus, there exists a scalar $\mu\in \F_q^\times$ such that 
\[ (S \Pi)^2 v^1 = \mu S \Pi v^2.\]
Then by the definition of the action of $\Pi$ on $D_1$, we have that 
\[(\tld{S} \tld{\Pi})^2 v^1_i = \lambda_i \mu \tld{S} \tld{\Pi} v^2_i\]
for all $i\in \Z$.

By Lemma \ref{lemma:circle}, there exists a nonzero $(c_i)_i$ in $\oplus_{i\in \Z} \ovl{\F}_p$ such that 
\[
\Big( \sum_i c_i \iota_i\Big) D_0(\rho) \subset \pi'.
\]
Assume that $\#(c_i)_i$ is minimal among such elements of $\oplus_{i\in \Z} \ovl{\F}_p$.
It suffices to show that $\#(c_i)_i$=1 by Lemma \ref{lemma:circle}.
By Lemma \ref{lemma:circle}, we can also assume that $c_0$ is nonzero.

Since $\sum_i c_i v_i^1$ and $\sum_i c_iv_i^2$ are in $\pi'$, then by the first paragraph, we have that
\[
\sum_i c_i ((\tld{S} \tld{\Pi})^2 v^1_i - \lambda_0\mu \tld{S}\tld{\Pi} v^2_i) = \sum_i (\lambda_i-\lambda_0) c_i\mu \tld{S}\tld{\Pi} v^2_i 
\]
is in $\pi'$, using that $\lambda_0\in \F_q$.
We see from Lemma \ref{lemma:circle} that 
\[
\Big( \sum_i c'_i \iota_i\Big) D_{0,\sigma_{\{0\} } }(\rho) \cap \pi' \neq 0
\]
for $c_i' = (\lambda_i-\lambda_0)c_i$.
Since the $\lambda_i \neq \lambda_0$ for $i\neq 0$ and $c_0\neq 0$, $\#(c'_i)_i = \#(c_i)-1$.
Since we assumed that $\#(c_i)_i$ is minimal, we must have that $\#(c_i)_i = 1$.
\end{proof}

We now complete the proof of Theorem \ref{thm:irr}.
By Lemma \ref{lemma:circle}, it suffices to show that $c$ in Lemma \ref{lemma:one} can be taken to be any element of $\ovl{\F}_p^\times$.
If $\pi'$ is a subrepresentation of $\pi$ over $\ovl{\F}_p$, this is clear.
Now assume that the $\F_q$-span of $(\lambda_i)_i$ is $\ovl{\F}_p$.
By Lemma \ref{lemma:circle}, $c\iota_j (D_0(\rho)^{\chi_1}) \subset \pi'$ for all $j\in \Z$.
By applying $\tld{\Pi}$ to $c\iota_j (D_0(\rho)^{\chi_1})$, we see that $c$ can be taken to be $c\lambda_j$ for all $j\in \Z$.
Since $(c\lambda_i)_i$ spans $\ovl{\F}_p$ over $\F_q$, we are done.
\end{proof}

Note that since $D_0$ is not admissible, any $\pi$ as in Theorem \ref{thm:con} is not admissible.
Taking $\lambda\in \prod_{i\in \Z} \F_q^\times$, Theorem \ref{thm:irr} implies Theorem \ref{thm:main} by taking the $\F_q$-model of $\pi$ constructed in Theorem \ref{thm:con}.
Since the endomorphisms of any such $\pi$ must contain $\ovl{\F}_p$, taking $(\lambda_i)_i$ to span $\ovl{\F}_p$ over $\F_q$ and restricting scalars of $\pi$ to $\F_q$, Theorem \ref{thm:irr} implies Theorem \ref{thm:schur}.

\bibliographystyle{amsalpha}
\bibliography{Nonadmissible}

\providecommand{\bysame}{\leavevmode\hbox to3em{\hrulefill}\thinspace}
\providecommand{\MR}{\relax\ifhmode\unskip\space\fi MR }
\providecommand{\MRhref}[2]{%
  \href{http://www.ams.org/mathscinet-getitem?mr=#1}{#2}
}
\providecommand{\href}[2]{#2}
\begin{thebibliography}{AHHV17}

\bibitem[AHHV]{AHHV2}
N.~Abe, G.~Henniart, F.~Herzig, and M.-F. Vign\'eras, \emph{Questions on mod
  $p$ representations of reductive $p$-adic groups}.

\bibitem[AHHV17]{AHHV}
\bysame, \emph{A classification of irreducible admissible {${\rm mod}\, p$}
  representations of {$p$}-adic reductive groups}, J. Amer. Math. Soc.
  \textbf{30} (2017), no.~2, 495--559. \MR{3600042}

\bibitem[Ber12]{Berger}
Laurent Berger, \emph{Central characters for smooth irreducible modular
  representations of {${\rm GL}_2({\bf Q}_p)$}}, Rend. Semin. Mat. Univ. Padova
  \textbf{128} (2012), 1--6 (2013). \MR{3076828}

\bibitem[BL94]{BL}
L.~Barthel and R.~Livn\'{e}, \emph{Irreducible modular representations of
  {${\rm GL}_2$} of a local field}, Duke Math. J. \textbf{75} (1994), no.~2,
  261--292. \MR{1290194}

\bibitem[BP12]{BP}
Christophe Breuil and Vytautas Pa{\v{s}}k{\=u}nas, \emph{Towards a modulo {$p$}
  {L}anglands correspondence for {${\rm GL}_2$}}, Mem. Amer. Math. Soc.
  \textbf{216} (2012), no.~1016, vi+114. \MR{2931521}

\bibitem[Bre03]{Br03}
Christophe Breuil, \emph{Sur quelques repr\'{e}sentations modulaires et
  {$p$}-adiques de {${\rm GL}_2(\bold Q_p)$}. {I}}, Compositio Math.
  \textbf{138} (2003), no.~2, 165--188. \MR{2018825}

\bibitem[Bre11]{Breuil}
\bysame, \emph{Diagrammes de {D}iamond et {$(\phi,\Gamma)$}-modules}, Israel J.
  Math. \textbf{182} (2011), 349--382. \MR{2783977}

\bibitem[Pa{\v{s}}04]{Paskunas}
Vytautas Pa{\v{s}}k{\=u}nas, \emph{Coefficient systems and supersingular
  representations of {${\rm GL}_2(F)$}}, M\'em. Soc. Math. Fr. (N.S.) (2004),
  no.~99, vi+84. \MR{2128381 (2005m:22017)}

\bibitem[Vig96]{V}
Marie-France Vign\'eras, \emph{Repr\'esentations {$l$}-modulaires d'un groupe
  r\'eductif {$p$}-adique avec {$l\ne p$}}, Progress in Mathematics, vol. 137,
  Birkh\"auser Boston, Inc., Boston, MA, 1996. \MR{1395151}

\end{thebibliography}

\end{document}